\documentclass[12pt,a4paper,twoside]{article}
\usepackage{amsmath}
\usepackage{amsthm}
\usepackage{amssymb}
\usepackage{verbatim}
\usepackage{booktabs}
\usepackage{caption}
\usepackage{subfig}
\usepackage{tabulary}
\usepackage{mathrsfs}
\usepackage{thmtools}
\usepackage{enumerate}

\usepackage[pdftex]{graphicx}

\newcommand{\spn}{{\mathrm{span}}}
\newcommand{\cnv}{{\mathrm{conv}}}
\newcommand{\cne}{{\mathrm{cone}}}
\newcommand{\ri}{{\mathrm{ri}}}
\newcommand{\inte}{{\mathrm{int}}}

\newcommand{\R}{{\mathbb{R}}}
\newcommand{\N}{{\mathbb{N}}}
\newcommand{\bU}{\bar{U}}
\newcommand{\bu}{\bar{u}}
\newcommand{\baralpha}{{\bar{\alpha}}}

\title{Bounds for $\alpha-$Optimal Partitioning of a Measurable Space Based on Several Efficient Partitions}
\author{Marco Dall'Aglio \\ LUISS University \\Rome, Italy \\ \texttt{mdallaglio@luiss.it} \and
Camilla Di Luca \\
LUISS University \\Rome, Italy \\
\texttt{cdiluca@luiss.it} }

\date{October 7, 2013}

\theoremstyle{plain}
\newtheorem{teo}{Theorem}
\newtheorem{defin}{Definition}
\newtheorem{prop}{Proposition}
\newtheorem{cor}{Corollary}
\newtheorem{lem}{Lemma}
\theoremstyle{remark}
\declaretheorem[style=definition]{example}
\newtheorem{rem}{Remark}

\begin{document}

\renewcommand\thmcontinues[1]{Continued}

\maketitle

\begin{abstract}
We provide a two-sided inequality for the $\alpha-$optimal partition value of a measurable space according to $n$ nonatomic finite measures. The result extends and often improves Legut (1988) since the bounds are obtained considering several partitions that maximize the weighted sum of the partition values with varying weights, instead of a single one.
\end{abstract}

\section{Introduction}

Let $(C,\mathcal{C})$ be a measurable space, $N=\{1,2,\ldots,n\}$, $n \in \N$  and let $\{\mu_i\}_{i \in N}$ be nonatomic finite measures defined on the same $\sigma-$algebra $\mathcal{C}$. Let $\mathscr{P}$ stand for the set of all measurable partitions $(A_1,\ldots,A_n)$ of $C$ ($A_i \in \mathcal{C}$ for all $i \in N$, $\cup_{i \in N}A_i = C$, $A_i \cap A_j = \emptyset$ for all $i \neq j$). Let $\Delta_{n-1}$ denote the $(n-1)$-dimensional simplex. For this definition and the many others taken from convex analysis, we refer to \cite{hl01}.
\begin{defin}
A partition $(A^*_1,\ldots,A^*_n) \in \mathscr{P}$ is said to be $\alpha-$optimal, for $\alpha = (\alpha_1, \ldots, \alpha_n) \in \inte \Delta_{n-1}$, if
\begin{equation} \label{v_maxmin}
v^{\alpha} := \min_{i \in N}\left\{ \frac{\mu_i(A_i^*)}{\alpha_i} \right\}
= \sup \left\{ \min_{i \in N} \left\{ \frac{\mu_i(A_i)}{\alpha_i} \right\} : (A_1,\ldots,A_n) \in \mathscr{P} \right\}.
\end{equation}
\end{defin}
This problem has a consolidated interpretation in economics. $C$ is a non-homogeneous, infinitely divisible good to be distributed among $n$ agents with idiosyncratic preferences, represented by the measures. A partition $(A_1,\ldots,A_n) \in \mathscr{P}$ describes a possible division of the cake, with slice $A_i$ given to agent $i \in N$. A satisfactory compromise between the conflicting interests of the agents, each having a relative claim $\alpha_i$, $i \in N$, over the cake, is given by the $\alpha-$optimal partition. It can be shown that the proposed solution coincides with the Kalai-Smorodinski solution for bargaining problems (See Kalai and Smorodinski \cite{ks75} and Kalai \cite{k77}).  When $\{\mu_i\}_{i \in N}$ are all probability measures, i.e..\ $\mu_i(C)=1$ for all $i \in N$, the claim vector $\alpha=(1/n,\ldots,1/n)$ describes a situation of perfect parity among agents. The necessity to consider finite measures stems from game theoretic extensions of the models, such as the one given in Dall'Aglio et al.\ \cite{dbt99}.

When all the $\mu_i$ are probability measures, Dubins and Spanier \cite{ds61} showed that if $\mu_i \neq \mu_j$ for some $i,j \in N$, then $v^{\alpha} > 1$. This bound was improved, together with the definition of an upper bound by Elton et al.\ \cite{ehk86}. A further improvement for the lower bound was given by Legut \cite{l88}. 

The aim of the present work is to provide further refinements for both bounds. We consider the same geometrical setting employed by Legut \cite{l88}, i.e. the partition range, also known as Individual Pieces Set (IPS) (see Barbanel \cite{b05} for a thorough review of its properties), defined as
$$\mathcal{R} := \{ (\mu_1(A_1), \ldots, \mu_n(A_n)) : (A_1, \ldots, A_n) \in \mathscr{P} \} \subset \mathbb{R}^n_+ .$$
Let us consider some of its features. The set $\mathcal{R}$ is compact and convex (see Lyapunov \cite{l40}). The supremum in \eqref{v_maxmin} is therefore attained. Moreover 
\begin{equation}\label{vmmax}
v^{\alpha} = \max
\{r \in \R_+ : (\alpha_1 r, \alpha_2 r, \ldots, \alpha_n r) \cap \mathcal{R} \neq \emptyset \}.
\end{equation}
So, the vector $(v^{\alpha} \alpha_1,\ldots,v^{\alpha} \alpha_n)$ is the intersection between the Pareto frontier of $\mathcal{R}$ and the ray $r \alpha= \{(r \alpha_1,\ldots,r \alpha_n): r \geq 0\}$.

To find both bounds, Legut locates the solution of the maxsum problem $
\sup \left\{ \sum_{i \in N} \mu_i(A_i): (A_1,\ldots,A_n) \in \mathscr{P} \right\}$
on the partition range. Then, he finds the convex hull of this point with the corner points $e^i=(0,\ldots,\mu_i(C),\ldots,0) \in \R^n$ ($\mu_i(C)$ is placed on the $i$-th coordinate) to find a lower bound, and uses a separating hyperplane argument to find the upper bound. We keep the same framework, but consider the solutions of several maxsum problems with weighted coordinates to find better approximations. Fix $\beta=(\beta_1,\ldots,\beta_n) \in \Delta_{n-1}$ and consider
\begin{equation}
\label{maxsumbeta}
\sum_{i \in N} \beta_i \mu_i(A_i^{\beta})=\sup \left\{ \sum_{i \in N} \beta_i \mu_i(A_i): (A_1,\ldots,A_n) \in \mathscr{P} \right\}.
\end{equation}
Let $\eta$ be a non-negative  finite-valued measure with respect to which each $\mu_i$ is absolutely continuous (for instance we may consider $\eta = \sum_{i \in N}{\mu_i}$). Then, by the Radon-Nikodym theorem for each $A \in \mathcal{C},$
\begin{equation*}
\mu_i(A) = \int_{A} f_i d \eta \quad \forall \; i \in N,
\end{equation*}
where $f_i$ is the Radon-Nikodym derivative of $\mu_i$
with respect to $\eta.$

Finding a solution for \eqref{maxsumbeta} is rather straightforward:
\begin{prop}(see \cite[Theorem 2]{ds61}, \cite[Theorem 2]{b00} \cite[Proposition 4.3]{d02})
Let $\beta \in \Delta_{n-1}$ and let
$B^{\beta}=(A^{\beta}_1 , \ldots, A^{\beta}_n)$ be an $n-$partition of $C$. If 
\begin{equation}\label{ppartition}
\beta_k f_k (x) \geq \beta_h f_h (x) \quad \text{for all }h,k \in N \text{ and for all } x \in A^{\beta}_k ,
\end{equation}
then $(A_1^{\beta},\ldots,A_n^{\beta})$ is optimal for \eqref{maxsumbeta}.
\end{prop}

\begin{defin}
Given $\beta \in \Delta_{n-1}$, an {\em efficient value vector (EVV)} with respect to $\beta$, $u^{\beta}=(u_1^{\beta}, \ldots, u_n^{\beta}),$ is defined by
$$u_i^{\beta} = \mu_{i} (A_i^{\beta}), \quad \mbox{ for each } i = 1, \ldots, n.$$
\end{defin}

The EVV $u^{\beta}$ is a point where the hyperplane
\begin{equation}
\label{supp_hyp}
\sum_{i \in N}{\beta_i x_i} = \sum_{i \in N}{\beta_i u_i}
\end{equation} touches the partition range
$\mathcal{R},$ so $u^{\beta}$ lies on the Pareto border of
$\mathcal{R}.$

\section{The main result}
As we will see later only one EVV is enough to assure a lower
bound, we give a general result for the case where several EVVs have
already been computed. We derive this approximation result
through a convex combination of these easily computable points in
$\mathcal{R},$ which lie close to $(\alpha_1 v^{\alpha},\ldots,\alpha_n v^{\alpha}).$ 

\begin{teo}\label{teo_main}
Consider $m \leq n$ linearly independent vectors
 $u^1, \ldots, u^m,$ where $u^{i}=(u_{i1},u_{i2},\ldots,u_{in})$, $i \in M:=\{1,2,\ldots,m\},$ is the EVV associated to $\beta^i$, $\beta^i=(\beta_{i1}, \ldots, \beta_{in}) \in \Delta_{n-1}$.
Let $U = (u^1, u^2, \ldots, u^m)$ be an $n \times m$ matrix and denote as $\bar{U}$ an $m \times m$ submatrix of $U$ such that
\begin{equation}
\label{detubar}
\det(\bU) \neq 0 \; .
\end{equation}
\begin{enumerate}[(i)]
\item 
\begin{equation}
\label{acone}
\alpha \in \cne(u^1,\ldots,u^m)
\end{equation} if and only if
\begin{equation} \label{det_condition}
\det(\bU)  \det(\bU_{\alpha i }) \geq 0 \qquad \mbox{for all } i \in M,
\end{equation}
where $\bU_{\alpha i }$ is the $m \times m$ matrix obtained by replacing the $i-$th column of $\bU$ with $\bar{\alpha} \in \mathbb{R}^m$, obtained from $\alpha$ by selecting the elements corresponding to the rows in $\bU$. 
Moreover,
\begin{equation}
\label{aricone}
\alpha \in \ri(\cne(u^1,\ldots,u^m))
\end{equation}
if and only if
\begin{equation} \label{det_strict}
\det(\bU)  \det(\bU_{\alpha i }) > 0 \qquad \mbox{for all } i \in M.
\end{equation}
\item For any choice of $u^1,\ldots,u^m$,
\begin{equation} \label{ubounds}
v^{\alpha} \leq \min_{i \in M} \cfrac{\sum_{j \in N} \beta_{ij} u_{ij}}{\sum_{j \in N} \beta_{ij} \alpha_j} \: .
\end{equation}
Moreover, if \eqref{det_condition} holds, then
\begin{equation} \label{lbounds}
\cfrac{1}{\sum_{i \in M} \sum_{j \in M} \baralpha_j \left[ \bU^{-1}\right]_{ij}} \leq v^{\alpha}
\end{equation}
where $\left[ \bar{U}^{-1}\right]_{ij}$ is the $ij$-th element of $\bar{U}^{-1}$.
\end{enumerate}  
\end{teo}

\begin{proof}
To prove $(i)$, suppose \eqref{det_condition} holds.
We show that $\cnv(u^1, \ldots, u^m) \cap r \alpha \neq \emptyset$, and therefore that \eqref{acone} holds, by verifying that the following system of linear equations in the variables $r$, $t_1,t_2,\ldots, t_m$
\begin{equation}\label{system}
\left\{
\begin{array}{l}
t_1 u_{11} + t_2 u_{21} + \ldots + t_m u_{m1} = \alpha_1 r \\
t_1 u_{12} + t_2 u_{22} + \ldots + t_m u_{m2} = \alpha_2 r \\
\vdots \\
t_1 u_{1n} + t_2 u_{2n} + \ldots + t_m u_{mn} = \alpha_n r \\
t_1 + t_2 + \ldots + t_m = 1
\end{array}
\right.
\end{equation}
has a unique solution $(r^*,t^*_1,\ldots,t^*_m)$ with $t^*_i \geq 0$, for $i \in M$. First of all, $\det(\bar{U}) \neq 0$ implies $\det(\bar{U}_{\alpha i^*}) \neq 0$ for at least an $i^* \in M$, otherwise all the EVVs would lie on the same hyperplane, contradicting the linear independence of such vectors. This fact and \eqref{det_condition} imply that the coefficient matrix has rank $m+1$ and its unique solution can be obtained by deleting the $n-m$ equations corresponding to the rows not in $\bU$.
Denote each column of $\bU$ as $\bu^i=(\bu_{i1},\ldots,\bu_{im})$, $i \in M$ and denote as  $\baralpha =(\baralpha_1,\ldots,\baralpha_m)$,  the vector obtained from $\alpha$ by selecting the same components as each $\bu^i.$ By Cramer's rule we have for each $i \in M$, 
\begin{align*}
t_i &= \frac{\det
\begin{pmatrix}
\bu_{11} & \bu_{21} & \ldots & \bu_{i-1,1} & 0 & \bu_{ i +1,1} & \ldots & \bu_{m1} & -\baralpha_1 \\
\bu_{11} & \bu_{22} & \ldots & \bu_{i-1,2} & 0 & \bu_{i +1,2} & \ldots & \bu_{m2} & -\baralpha_2 \\
\vdots & \vdots & \vdots & \vdots & \vdots & \vdots & \vdots & \vdots & \vdots \\
\bu_{1m} & \bu_{2m} & \ldots & \bu_{i-1,m} & 0 & \bu_{i +1, m} & \ldots & \bu_{mm} & -\baralpha_m \\
1 & 1 & \ldots & 1 & 1 & 1 & \ldots & 1 & 0
\end{pmatrix}
}{\det
\begin{pmatrix}
\bu_{11} & \ldots & \bu_{m1} & -\baralpha_1 \\
\vdots & \vdots & \vdots & \vdots \\
\bu_{1m} & \ldots & \bu_{mm} & -\baralpha_m \\
1 & \ldots & 1 & 0
\end{pmatrix}
} 
\end{align*}
\begin{align*}
&= \frac{ (-1)^{i + (m+1)} \det
\begin{pmatrix}
\bu_{11} & \bu_{21} & \ldots & \bu_{ i-1, 1} & \bu_{i +1, 1} & \ldots & \bu_{m1} & -\baralpha_1 \\
\bu_{12} & \bu_{22} & \ldots & \bu_{i-1, 2} & \bu_{i +1, 2} & \ldots & \bu_{m2} & -\baralpha_2 \\
\vdots & \vdots & \vdots & \vdots & \vdots & \vdots & \vdots & \vdots \\
\bu_{1m} & \bu_{2m} & \ldots & \bu_{i-1, m} & \bu_{i +1, m} & \ldots & \bu_{mm} & -\baralpha_m 
\end{pmatrix}
}{\sum_{j \in M} (-1)^{(m+1) + j} \det
\begin{pmatrix}
\bu_{11} & \ldots & \bu_{j-1, 1} & \bu_{j+1, 1} & \ldots & \bu_{m1} & -\baralpha_1 \\
\vdots & \vdots & \vdots & \vdots & \vdots & \vdots & \vdots \\
\bu_{1m} & \ldots & \bu_{j-1, m} & \bu_{j+1, m} & \ldots & \bu_{mm} & -\baralpha_m 
\end{pmatrix}
} \\
&= \cfrac{(-1)^{2m + 2} \det(\bU_{\alpha i})}{\sum_{j \in M} (-1)^{2m+2} \det(\bU_{\alpha j})} 
= \cfrac{\det(\bU_{\alpha i})}{\sum_{j \in M} \det(\bU_{\alpha j})} \geq 0
\end{align*}
since by \eqref{det_condition} either a determinant is null or it has the same sign of the other determinants. 
If \eqref{det_strict} holds, then $t_i > 0$ for every $i \in M$ and \eqref{aricone} holds.

Conversely, each row of $U$ not in $\bU$ is a linear combination of the rows in $\bU$. Therefore, each point of $\spn(u^1,\ldots,u^m)$ is identified by a vector $x \in \R^m$ whose components correspond to the rows in $\bU$, while the other components are obtained by means of the same linear combinations that yield the rows of $U$ outside $\bU$.  

Let $\bU_{-j}$ denote the $m \times (m-1)$ matrix obtained from the matrix $\bU$ without $\bu^j$, $ j \in M$.
Consider a hyperplane $H_{-j}$ in $\spn(u^1,\ldots,u^m)$ through the origin and $m-1$ EVVs
$$H_{-j}:= \det(x,\bU_{-j})=0,$$
where $j \in N$. 
If $\alpha \notin \cne(u^1,\ldots,u^m)$, 
when we separate the subspace through $H_{-j}$, for all $j \in N$
either the ray $r \alpha$ is coplanar to $H_{-j}$, i.e., 
\begin{equation}\label{coplanar}
\det(\baralpha, \bU_{-j})= 0,
\end{equation}
or $\baralpha$ and $\bu^j$ lie in the same half-space, i.e.,
\begin{equation}\label{equalities}
\det(\baralpha, \bU_{-j})  \det(\bu^j, \bU_{-j}) >0.
\end{equation} 
Moving the first column to the $j$-th position in all the matrices above, we get \eqref{det_condition}. In case \eqref{aricone} holds, only the inequalities in \eqref{equalities} are feasible and \eqref{det_strict} holds.

To prove $(ii)$, consider, for any $i \in N$, the hyperplane \eqref{supp_hyp} that intersects the ray $r \alpha$ at the point $(\bar{r}_i \alpha_1,\ldots,\bar{r}_i \alpha_n)$, with $$
\bar{r}_i= \cfrac{\sum_{j \in N} \beta_{ij} u_{ij}}{\sum_{j \in N} \beta_{ij} \alpha_j} \: .
$$
Since $\mathcal{R}$ is convex, the intersection point is not internal to $\mathcal{R}$. 
So, $\bar{r}_i  \geq
v^{\alpha} $ for $i  \in M$, and, therefore, $\min_{i \in M} \bar{r}_i \geq v^{\alpha}$.

We get the lower bound for $v^{\alpha}$ as solution in $r$ of the  system \eqref{system}.
By Cramer's rule, 

\begin{align*}
r^* &= \frac{\det
\begin{pmatrix}
\bu_{11} & \bu_{21} & \ldots & \bu_{m1} & 0 \\
\bu_{12} & \bu_{22} & \ldots & \bu_{m2} & 0 \\
\vdots & \vdots & \vdots & \vdots & \vdots \\
\bu_{1m} & \bu_{2m} & \ldots & \bu_{mm} & 0 \\
1 & 1 & \ldots & 1 & 1
\end{pmatrix}
}{\det
\begin{pmatrix}
\bu_{11} & \bu_{21} & \ldots & \bu_{m1} & -\baralpha_1 \\
\bu_{12} & \bu_{22} & \ldots & \bu_{m2} & -\baralpha_2 \\
\vdots & \vdots & \vdots & \vdots & \vdots \\
\bu_{1m} & \bu_{2m} & \ldots & \bu_{mm} & -\baralpha_m \\
1 & 1 & \ldots & 1 & 0
\end{pmatrix}
} 
= \frac{\det(\bU)}{\det
\begin{pmatrix}
0 & 1 & 1 & \ldots & 1 \\
-\baralpha_1 & \bu_{11} & \bu_{21} & \ldots & \bu_{m1} \\
- \baralpha_2 & \bu_{12} & \bu_{22} & \ldots & \bu_{m2} \\
\vdots & \vdots & \vdots & \vdots & \vdots \\
-\baralpha_m & \bu_{1m} & \bu_{2m} & \ldots & \bu_{mm}
\end{pmatrix}
} \\
&= \cfrac{\det(\bU)}{\sum_{i \in M}{\sum_{j \in M}{(-1)^{i
+j} \baralpha_j \det(\bU_{i j})}}}
= \cfrac{1}{\sum_{i \in M} \sum_{j \in M} \baralpha_j \left[ \bU^{-1}\right]_{ij}} ,
\end{align*}
where $\det(\bU_{ij})$ is the $ij$-th minor of $\bU$.
The second equality derives by suitable exchanges of rows and columns in the denominator matrix: 
In fact, swapping the first rows and columns of the matrix leaves the determinant unaltered.
The last equality derives by dividing each row $i$ of $\bU$ by $\alpha_i$, $i \in M$. 
Finally, by \eqref{vmmax} we have $r^* \leq v^{\alpha}$.
\qedhere
\end{proof}

\begin{rem}
The above result shows that whenever $\det(\bU_{\alpha i})=0$, then $t_i=0.$ Therefore, the corresponding EVV $u^i$ is irrelevant to the formulation of the lower bound and can be discarded. We will therefore keep only those EVV that satisfy \eqref{det_strict} and will denote them as the {\em supporting EVVs} for the lower bound.
\end{rem}
In the case $m=n$ we have

\begin{cor}\label{cor_mequalsn}
Suppose that there are $n$ vectors $u^1, \ldots, u^n,$ where $u^{i}$, $i \in N$, is the EVV associated to $\beta^i$, $\beta^i=(\beta_{i1}, \ldots, \beta_{in}) \in \Delta_{n-1}$. If $U=(u^1,\ldots,u^n)$, $\det(U) \neq 0$ and, for all $i \in N$ 
\begin{equation} \label{det_mequalsn}
\det(U) \cdot \det(U_{\alpha i }) \geq 0,
\end{equation}
where $U_{\alpha i }$ is the $n \times n$ matrix obtained by replacing $u^i$ with $\alpha$ in $U$, then
\begin{equation} \label{bounds_mequalsn}
\cfrac{1}{\sum_{i \in N} \sum_{j \in N} \alpha_j \left[ U^{-1}\right]_{ij}} \leq v^{\alpha} \leq \min_{i \in N} \cfrac{\sum_{j \in N} \beta_{ij} u_{ij}}{\sum_{j \in N} \beta_{ij} \alpha_j}
\end{equation}
where $\left[ U^{-1}\right]_{ij}$ is the $ij$-th element of $U^{-1}$.
\end{cor}

We next consider two further corollaries that provide bounds in case only one EVV is available.
The first one works with an EVV associated to an arbitrary vector $\beta \in \Delta_{n-1}.$

\begin{cor}
\label{oneEVV}
(\cite[Proposition 3.4]{dd_arxiv}) Let $\mu_1, \ldots, \mu_n$ be finite measures and let $u=(u_1,u_2,\ldots,u_n)$ be the EVV corresponding to $\beta \in \Delta_{n-1}$ such that 
\begin{equation}\label{umax}
\alpha_j^{-1} u_j = \max_{i \in N} \alpha_i^{-1} u_i.
\end{equation} 
Then,
\begin{equation} 
\cfrac{u_j}{\alpha_j + \sum_{i \neq j}\left[ \mu_i^{-1} (C) (\alpha_i u_j - \alpha_j u_i) \right]} \leq v^{\alpha} \leq \cfrac{\sum_{i \in N} \beta_i u_i}{\sum_{i \in N} \alpha_i u_i}.
\end{equation}
\end{cor}

\begin{proof}
Consider the corner points of the partition range
$$
e^i = (0, \ldots, \mu_i(C), \ldots, 0) \in \R^n \; ,
$$
where $\mu_i(C)$ is placed on the $i$-th coordinate ($i \in N$), and the matrix
$U=(e^1, \ldots, e^{j-1}, u, e^{j+1}, \ldots, e^n)$, where $u$ occupies the $j$-th position.
Now
\begin{gather*}
det(U)=u_j \prod _{i \in N \setminus\{j\}} {\mu_i (C)}>0
\\
det(U_{\alpha j})=\alpha_j \prod _{i \in N \setminus\{j\}} {\mu_i (C)} >0
\end{gather*}
and, for all $i \in N \setminus \{j\}$,
$$
det(U_{\alpha i})= (\alpha_i u_j - \alpha_j u_i) \prod_{k \in N \setminus \{i, j\}}{\mu_k}(C) \; ,
$$ 
which is positive by \eqref{umax}. Therefore, $U$ satisfies the hypotheses of Corollary
\ref{cor_mequalsn}. Since $U$ has inverse
$$
U^{-1}=
\begin{pmatrix}
\frac{1}{\mu_1(C)} & 0 & \cdots & -\frac{u_1}{\mu_1(C) u_j} & \cdots & 0
\\
0 &\frac{1}{\mu_1(C)}  & \cdots & -\frac{u_2}{\mu_2(C) u_j} & \cdots & 0
\\
\vdots & \vdots & \ddots & \vdots & \ddots & \vdots
\\
0 & 0 & \cdots & \frac{1}{ u_j} & \cdots & 0
\\
\vdots & \vdots & \ddots & \vdots & \ddots & \vdots
\\
0 & 0 & \cdots & -\frac{u_n}{\mu_n(C) u_j} & \cdots & \frac{1}{\mu_n(C)}
\end{pmatrix} \; ,
$$
the following  lower bound is guaranteed for $v^{\alpha}$:
$$
v^{\alpha} \geq r^* = \cfrac{u_j}{ \alpha_j + \sum_{i \in N \setminus \{j\} } \left[ \mu_i^{-1} (C) (\alpha_i u_j - \alpha_j u_i) \right] } .
$$
The upper bound is a direct consequence of Theorem 
\ref{teo_main}.
\qedhere
\end{proof}
In case all measures $\mu_i$, $i \in N,$ are normalized to one and the only EVV considered is the one corresponding to $\beta=(1/n, \ldots, 1/n)$, we obtain Legut's result.

\begin{cor}\label{legut_result}
(\cite[Theorem 3]{l88}) Let $\mu_1, \ldots, \mu_n$ be probability measures and let $u=(u_1,u_2,\ldots,u_n)$ be the EVV corresponding to $\beta=(1/n, \ldots, 1/n)$. Let $j \in N$ be such that \eqref{umax} holds.
Then,
\begin{equation} \label{legut_bounds}
\cfrac{u_j}{u_j - \alpha_j(K-1)} \leq v^{\alpha} \leq \sum_{i \in N}{u_i},
\end{equation}
where $K = \sum_{i \in N} u_i$.  
\end{cor}
\begin{proof}
Simply apply Corollary \ref{oneEVV} with $\mu_i(C)=1$, for all $i \in N$ and $\beta=(1/n, \ldots, 1/n)$. Then

\begin{align*}
v^{\alpha} \geq r^* &= \cfrac{u_j}{\alpha_j + \sum_{i \neq j}(\alpha_i u_j - \alpha_j u_i)}
= \cfrac{u_j}{u_j - \alpha_j(K-1)},
\end{align*}
where $K= \sum_{i \in N}{u_i}$.
Finally, by Theorem \ref{teo_main} we have 
$$v^{\alpha} \leq  \cfrac{\sum_{i \in N} \left(\frac{1}{n} \, u_i \right)}{\frac{1}{n}\sum_{i \in N} \alpha_i} = \sum_{i \in N} u_i . $$ 
\qedhere
\end{proof}

It is important to notice that the lower bound provided by Theorem \ref{teo_main} certainly improves on Legut's lower bound only when one of the EVVs forming the matrix $U$ is the one associated to $\beta=(1/n, \ldots, 1/n).$

\begin{example}[label=exa:cont]
We consider a $[0,1]$ good that has to be divided among three agents with equal claims, $\alpha = (1/3, 1/3, 1/3),$ and preferences given as density functions of probability measures
$$
f_1(x)=1 \qquad f_2(x)=2x \qquad f_3(x)= 30 x (1-x)^4  \qquad x \in [0,1] \; ,
$$ 
$f_3$ being the density function of a $\mathrm{Beta}(2,5)$ distribution. The preferences of the players are not concentrated (following Definition 12.9 in Barbanel \cite{b05}) and therefore there is only one EVV associated to each $\beta \in \Delta_{2}$ (cfr.\ \cite{b05}, Theorem 12.12)

\begin{figure}[h!]
  \caption{The density functions in Example \ref{exa:cont}. Agent 1: tiny dashing; Agent 2: large dashing; Agent 3: continuous line.}
  \centering
    \includegraphics[width=0.6\textwidth]{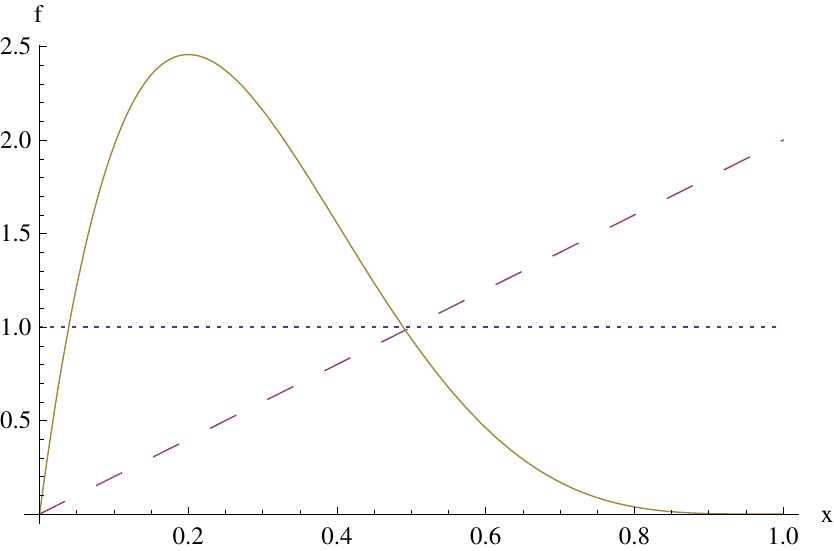}
\end{figure} 
The EVV corresponding to $\beta_{eq}=(1/3,1/3,1/3)$ is $u_{eq}=(0.0501, 0.75, 0.8594)$. Consequently, the bounds provided by Legut are
$$
1.3437 \leq v^{\alpha} \leq 1.6594.
$$
Consider now two other EVVs
$$
u^1= (0.2881, 0.5556, 0.7761)
\qquad
u^2 = (0.25, 0.9375, 0)
$$
corresponding to $\beta^1=(0.4, 0.3, 0.3)$ and $\beta^2=(0.3, 0.6, 0.1)$, respectively. The matrix $U=(u_1,u_2,u_{eq})$ satisfies the hypotheses of Theorem \ref{teo_main} and the improved bounds are
$$
1.4656 \leq v^{\alpha} \leq 1.5443.
$$
\end{example}

\section{Improving the bounds}
The bounds for $v^{\alpha}$ depend on the choice of the EVVs that satisfy the hypotheses of Theorem \ref{teo_main}. Any new EVV yields a new term in the upper bound. Since we consider the minimum of these terms, this addition is never harmful. Improving the lower bound is a more delicate task, since we should modify the set of supporting EVVs for the lower bound, i.e.\ those EVVs that include the ray $r \alpha$ in their convex hull. When we examine a new EVV we should verify whether replacing an EVV in the old set will bring to an improvement.

A brute force method would require us to verify whether conditions \eqref{det_condition} are verified with the new EVV
in place of $\alpha.$ Only in this case we have a guarantee that the new EVV will not make the bound worst. Then, we should verify \eqref{det_condition} again, with the new EVV
replacing one of the $m$ EVVs, in order to find the EVV from the old set to replace.
However, in the following proposition we propose
a more efficient condition for improving the bounds, 
by which we simultaneously verify that the new EVV belongs to the convex hull of the $m$ EVVs
and detect the vector to replace.

For any couple $j,k \in M$ denote as $\bU_{*j-k}$ the $m \times (m-1)$ matrix obtained from $\bU$ by replacing column $j$ with $u^*$ and by deleting column $k$.
\begin{teo}
\label{teo_imp}
Let $u^*,u^1,u^2,\ldots,u^m$ be $m + 1$ EVVs with the last $m$ vectors satisfying conditions \eqref{detubar} and \eqref{det_strict}. If there exists  $j \in M$ such that
\begin{equation}
\label{eq_detstar}
\det(\bar{\alpha},\bU_{*j-k}) \det(\bu^j ,\bU_{*j-k}) \leq 0 \qquad \mbox{for every }k \in M \setminus \{j\}
\end{equation}
then
\begin{gather}
\label{uscone}
u^* \in \cne(u^1,\ldots,u^m)
\\
\label{aconeminj}
\alpha \in \cne(\{u^i\}_{i \neq j}, u^*)
\end{gather}
Moreover, if all the inequalities in \eqref{eq_detstar} are strict then in both \eqref{uscone} and \eqref{aconeminj} the vectors belong to the relative interior of the respective cones.
\end{teo}

\begin{proof}
Before proving the individual statements, we sketch a geometric interpretation for condition \eqref{eq_detstar}. As in Theorem \ref{teo_main} we restrict our analysis to the subspace $\spn(u^1,\ldots,u^m)$. For any $k \in M \setminus \{j\}$, the hyperplanes
$$
H_{*j-k}=\{x \in \R^m: \det(x,\bU_{*j-k})=0\}
$$
should separate $u^j$ and $\alpha$ (strictly if all the inequalities in \eqref{eq_detstar} are strict) in the subspace $\spn(u^1,\ldots,u^m)$.

To prove \eqref{uscone}, argue by contradiction and suppose $u^* \notin \cne(u^1,\ldots,u^m).$ Then, for any $j \in M,$ there must exist a $k \neq j$ such that the hyperplane $H_{*j-k}$ passes through all the EVV (including $u^*$) but $u^j$ and $u^k$, and supports $\cne(u^1,\ldots,u^m).$ Therefore, $\alpha$ and $u^j$ belong to the same strict halfspace defined by $H_{*j-k}$, contradicting \eqref{eq_detstar}.

To show the existence of such hyperplane consider the hyperplane $H_M$ in $\spn(u^1,\ldots,u^m)$ passing through $u^1,\ldots,u^m$ and denote with $u_M^*$ the intersection of $r u^*$ with such hyperplane. Restricting our attention to the points in $H_M$, the vectors $u^1,\ldots,u^m$ form a simplicial polyhedron with $u^*_M \notin \cnv(u^1,\ldots,u^m)$. There must thus exist a $k \neq j$ such that the $(m-2)$-dimensional hyperplane $H_{M-jk}$ in $H_M,$ passing through $u^*_M$ and $\{u^i\}_{i \notin \{j,k\}}$, supports $\cnv(u^1,\ldots,u^m)$ and contains $u^j$ (and $\alpha$) in one of its strict halfspaces (see Appendix.) If we now consider the hyperplane in $\spn(u^1,\ldots,u^m)$ passing through the origin and $H_{M-jk}$ we obtain the required hyperplane $H_{*j-k}.$

To prove \eqref{aconeminj} we need some preliminary results. First of all, under \eqref{eq_detstar},
\begin{equation}
\label{eq_*jnonzero}
\det(\bU_{*j}) \neq 0
\end{equation}
Otherwise, $u^*$ would be coplanar to $H_{-j}$ and any hyperplane $H_{*j-k}$, $k \neq j$, would coincide with it. In such case the separating conditions \eqref{eq_detstar} would not hold. Moreover, there must exist some other $h \neq j$ for which
\begin{equation}
\label{eq_*hnonzero}
\det(\bU_{*h}) \neq 0
\end{equation}
otherwise, $u^*$ would coincide with $u^j$, making the result trivial.

We also derive an equivalent condition for \eqref{eq_detstar}. Let $\bU_{aj,bk}$ be the $m \times m$ matrix obtained from $\bU$ by replacing vectors $u^j$ and $u^k$ by some other vectors, say $u^a$ and $u^b$, respectively. If we move the first column to the $k$-th position, \eqref{eq_detstar} becomes
$$
\det(\bU_{*j,\alpha k}) \det(\bU_{*j,jk}) \leq 0 \qquad \mbox{for every }k \neq j
$$
Switching positions $j$ and $k$ in the second matrix we get $\det(\bU_{*j,jk}) = - \det(\bU_{*k})$
and therefore
\begin{equation}
\label{eq_dsequiv}
\det(\bU_{*j,\alpha k}) \det(\bU_{*k}) \geq 0 \qquad \mbox{for every } k \neq j \: .
\end{equation}
From  \eqref{eq_*jnonzero}, \eqref{eq_*hnonzero} and from part $(i)$ of the present Theorem, we derive
$$
\det(\bU_{*j}) \det(\bU_{*h}) > 0
$$
and therefore, \eqref{eq_dsequiv} yields
\begin{multline}
\label{eq_alphaconv*}
\det(\bU_{*j,\alpha k}) \det(\bU_{*j}) =
\det(\bU_{*j,\alpha k}) \det(\bU_{*h}) \det(\bU_{*j}) \det(\bU_{*h})^{-1} \geq 0
\\
\mbox{for any }k \neq j
\end{multline}
Condition \eqref{eq_alphaconv*} and Theorem \ref{teo_main} allow us to conclude that $\alpha \in \cnv(\{u^k\}_{k \neq j},u^*)$.

Regarding the last statement of the Theorem, we have already shown \eqref{eq_*jnonzero}. Moreover, if \eqref{eq_detstar} hold with a strict inequality sign for any $k \neq j$, then $\det(\bU_{*k}) \neq 0$ for the same $k$ and $u^* \in \ri(\cne(u^1,\ldots,u^m))$. Similarly, \eqref{eq_alphaconv*} would hold with strict inequality signs and $\alpha \in \ri(\cne(\{u^k\}_{k \neq j},u^*))$.
\qedhere
\end{proof}
\begin{rem}
If \eqref{eq_detstar} holds, we get
not only that $r \alpha$ intersects the convex hull of the $m$ EVVs $\{u^i\}_{i \in M \setminus \{j\}} \cup \{u^*\}$, but also that the ray $r u^*$ intersects the convex hull of the $m$ EVVs $u^1, \ldots, u^m$. We can therefore replace $u^j$ with $u^*$ in the set of supporting EVVs for the lower bound. If the test fails for each $j \in M$, we discard $u^*$ we keep the current lower bound (with its supporting EVVs).     

In case \eqref{eq_detstar} holds with an equality sign for some $k$, conditions \eqref{eq_*jnonzero} and \eqref{eq_alphaconv*} together imply $\det(\bU_{*j,\alpha k})=0$. Therefore, we could discard $u^k$ from the set of supporting EVVs for the lower bound.
\end{rem}

\begin{example}[continues=exa:cont] We consider a list of 1'000 random vectors in $\Delta_2$ and, starting from the identity matrix, we iteratively pick each vector in the list. If this satisfies condition \eqref{eq_detstar}, then the matrix $U$ is updated. The update occurs 9 times and the resulting EVVs which generate the matrix $U$ are 
\begin{gather*}
u^1 = (0.5144, 0.5663, 0.3447) \\
u^2 =(0.4858, 0.5462, 0.4410) \\
u^3= (0.4816, 0.3910, 0.6551)
\end{gather*}
corresponding, respectively, to  
\begin{gather*}
\beta^1 = (0.4273, 0.2597, 0.3130)
\\
\beta^2 = (0.4524, 0.3357, 0.2119)
\\
\beta^3 = (0.4543, 0.3450, 0.2007)
\end{gather*}
Correspondingly, the bounds shrink to
$$
1.4792 \leq v^{\alpha} \leq 1.4898.
$$

\end{example}
The previous example shows that updating the matrix $U$ of EVVs through a random selection of the new candidates is rather inefficient, since it takes more than 100 new random vectors, on average, to find a valid replacement for vectors in $U$.

A more efficient way method picks the candidate EVVs through some accurate choice of the corresponding values of $\beta$. In \cite{dd_arxiv} a subgradient method is considered to find the value of $v^{\alpha}$ up to any specified level of precision. In that algorithm, Legut's lower bound is used, but this can be replaced by the lower bound suggested by Theorem \ref{teo_main}.

\begin{example}[continues=exa:cont]
Considering the improved subgradient algorithm, we obtain the following sharper bounds
$$
1.48768 \leq v^{\alpha} \leq 1.48775
$$
after 27 iterations of the algorithm in which, at each repetion, a new EVV is considered.
\end{example}

\subsection*{Acknowledgements}
The authors would like to thank Vincenzo Acciaro and Paola Cellini for their precious help.

\section*{Appendix}
The proof of \eqref{uscone} in Theorem \ref{teo_imp} is based on the following Lemma. This is probably known and too trivial to appear in a published version of the present work. However, we could not find an explicit reference to cite it. Therefore, we state and prove the result in this appendix
\begin{lem}
 Consider $n$ affinely independent points $u^1,\ldots,u^n$ in $\R^{n-1}$ and $u^* \notin \cnv(u^1,\ldots,u^n)$. For each $j \in N$ there must exist a $k \neq j$ such that the hyperplane passing through $u^*$ and $\{u^i\}_{i \notin \{j,k\}}$ supports $\cnv(u^1,\ldots,u^m)$ and has $u^j$ in one of its strict halfspaces.
\end{lem}
\begin{proof}
Suppose the thesis is not true. Then, for any $h,\ell \in N$, the hyperplane passing through $u^*$ and $\{u^i\}_{i \in N \setminus\{h,\ell\}}$ will strictly separate the remaining points $u^h$ and $u^{\ell}$.

Fix now $j \in N$ and consider $H_{-j}$, the hyperplane passing through the points $\{u^i\}_{i \in N \setminus \{j\}}$. Also denote as $u^j_*$ the intersection between $H_{-j}$ and the line joining $u^*$ and $u^j$. Clearly $u^j_* \notin \cnv (\{u^i\}_{i \in N \setminus \{j\}})$. Therefore, for any $k \in N \setminus \{j\}$, the hyperplane $H_{-jk}$ in $H_{-j}$ passing through $\{u^i\}_{i \in N \setminus \{j,k\}}$ will strictly separate $u^k$ and $u^j_*$. Consequently,  $u^j_*$ should simultaneously lie in the halfspace of $H_{-jk}$ not containing $\cnv (\{u^i\}_{i \in N \setminus \{j\}})$, and in the cone formed by the other hyperplanes $H_{-jh}$, $h \in N \setminus \{j,k\}$ and not containing $\cnv (\{u^i\}_{i \in N \setminus \{j\}})$. A contradiction.

\qedhere
\end{proof}

\end{document}